\documentclass{article}
\usepackage{amsmath}
\usepackage{amsfonts}
\usepackage{theorem}
\usepackage{float}
\usepackage{times}
\usepackage{graphicx}
\usepackage{color}
\usepackage{amssymb}
\usepackage{url}

\newtheorem{theorem}{Theorem}

\newtheorem{conclusion}{Conclusion}

\newtheorem{proposition}{Proposition}
\theorembodyfont{\upshape}
\newtheorem{example}{Example}

\newenvironment{proof}[1][Proof]{\noindent\textbf{#1.} }{\ \rule{0.5em}{0.5em}}
\input{tcilatex}
\begin{document}

\title{Loci of Points Inspired by Viviani's Theorem }
\author{Elias Abboud \\
%EndAName
{\small Faculty of education}, {\small Beit Berl College, Doar Beit Berl
44905, Israel }\\
{\small eabboud@beitberl.ac.il}}
\maketitle

Viviani (1622-1703), who was a student and assistant of Galileo, discovered
that equilateral triangles satisfy the following property: the sum of the
distances from the sides of any point inside an equilateral triangle is
constant.

Viviani's theorem can be easily proved by using areas. Joining a point $P$
inside the triangle to its vertices divides it into three parts. The sum of
their areas will be equal to the area of the original one. Therefore, the
sum of the distances from the sides will be equal to the altitude, and the
theorem follows.

Given a triangle (which includes both boundary and inner points), $\Delta ,$
and a positive constant, $k,$ we shall consider the following questions:

\textbf{Question 1}.\textbf{\ }\textit{What is the locus }$T_{k}(\Delta )$%
\textit{\ of points }$P$\textit{\ in }$\Delta $\textit{\ such that the sum
of the distances of }$P$\textit{\ to each of the sides of }$\Delta $\textit{%
\ equals }$k$\textit{? }

\textbf{Question 2}. \textit{What is the locus }$S_{k}(\Delta )$\textit{\ of
points }$P$\textit{\ in the plane such that the sum of the squares of the
distances of }$P$\textit{\ to each of the sides of }$\Delta $\textit{\
equals }$k$\textit{?}

By Viviani's theorem, the answer to the first question is straightforward
for equilateral triangles: If $k$ equals the altitude of an equilateral
triangle, then the locus of points is the entire triangle. On the other
hand, if $k$ is not equal to the altitude of the equilateral triangle then
the locus of points is empty.

Since the distance function is linear with two variable and because of
squaring the distances, the locus of points in the second question, if not
empty, is a quadratic curve. In our case, we shall see that this quadratic
curve is always an ellipse. This fact will be the basis of a new
characterization of ellipses.

We start first with the locus of points which have constant sum of distances
to each of the sides of a given triangle and consider separately the cases
of isosceles and scalene triangles. Then, we turn to the locus of points
which have constant sum of squared distances to each of the sides of a given
triangle. We prove the main theorem which gives a characterization of
ellipses and conclude by discussing minimal sum of squared distances.

\section*{Constant sum of distances}

Samelson \cite[p. 225]{Sam} gave a proof of Viviani's theorem that uses
vectors and Chen \& Liang \cite[p. 390-391]{CL} used this vector method to
prove a converse: if inside a triangle there is a circular region in which
the sum of the distances from the sides is constant, then the triangle is
equilateral. In \cite{Abb}, this converse is generalized in the form of two
theorems that we will use in addressing Question 1. To state these theorems,
we need the following terminology: Let $\mathcal{P}$ be a polygon consisting
of both boundary and interior points. Define the \textit{distance sum
function} $\mathcal{V}:\mathcal{P}\rightarrow \mathbb{R},$ where for each
point $P\in \mathcal{P}$ $\ $the value $\mathcal{V}(P)$ is defined as the
sum of the distances from $P$ to the sides of $\mathcal{P}.$\ \ 

\begin{theorem}
\label{equilateral}Any triangle can be divided into parallel line segments
on which $\mathcal{V}$ is constant. Furthermore, the following conditions
are equivalent:

\begin{itemize}
\item $\mathcal{V}$ is constant on\textit{\ }the\textit{\ triangle }$\Delta
. $

\item \textit{There are three non-collinear points, inside the triangle, at
which }$\mathcal{V}$ takes the same value\textit{.}

\item \textit{\ }$\Delta $ is \textit{equilateral. }
\end{itemize}
\end{theorem}

\begin{theorem}
\label{convex polygon}(a) Any convex polygon $\mathcal{P}$ can be divided
into parallel line segments, on which $\mathcal{V}$ is constant. \ 

(b) If $\ \mathcal{V}$ takes equal values at three non-collinear points,
inside a convex polygon, then $\mathcal{V}$ is constant on $\mathcal{P}$ .
\end{theorem}

The discussion in \cite[p. 207-210]{Abb} lays out a connection between
linear programming and Theorems \ref{equilateral} and \ref{convex polygon}.
Theorem \ref{equilateral} is proved explicitly by defining a suitable linear
programming problem and Theorem \ref{convex polygon} is proved by means of
analytic geometry, where the distance sum function $\mathcal{V}$ is computed
directly and shown to be linear in two variables. We shall apply these
theorems to find the loci $T_{k}(\Delta )$ for isosceles and scalene
triangles $\Delta .$

\subsection*{Isosceles triangle}

For isosceles triangles we exploit their reflection symmetry to find
directly the line segments on which $\mathcal{V}$ is constant, and for
scalene triangles we apply any of the above theorems to show that $\mathcal{V%
}$ is constant on certain parallel line segments. To find the direction of
these line segments we compute explicitly the equation of $\mathcal{V}$. The
idea is illustrated through an example.

Since an isosceles triangle has a reflection symmetry across the altitude,
we conclude that; if the sum of distances of point $P$ from the sides of the
triangle\ is $k$, then the reflection point $P^{\prime }$ across the
altitude satisfies the same property. For such triangles we have the
following proposition.

\begin{proposition}
If $\Delta $ is a non-equilateral isosceles triangle and $k$ ranges between
the lengths of the smallest and the largest altitudes of the triangle, then $%
T_{k}(\Delta )$ is a line segment, whose end points are on the boundary of
the triangle, parallel to the base.
\end{proposition}

\begin{proof}
In Fig. \ref{fig1}, if the line segment $DE$ is parallel to the base $BC$
then, when $P$ moves along $DE$, the length $a$ remains constant and $%
b+c=DP\sin \alpha +PE\sin \alpha =DE\sin \alpha ,$ which is also constant.
Hence, $a+b+c=k$ is constant on the segment $DE.$ By Theorem \ref%
{equilateral}, there are no other points in $\Delta $ with distance sum $k$
unless the triangle is equilateral.

If $a$ approaches zero then the length of $DE$ approaches the length of the
base $BC,$ and $k=a+DE\sin \alpha $ approaches the length of the altitudes
from $B$ (or $C).$ On the other hand, if the length of $DE$ approaches zero
then $k$ approaches the length of the altitude from $A$. Hence, $%
T_{k}(\Delta )$ is defined whenever $k$ ranges between the lengths of the
smallest and the largest altitudes of the triangle.
\end{proof}

\FRAME{ftbphFU}{3.9807in}{2.6472in}{0pt}{\Qcb{The locus of points is a line
segment parallel to the base. }}{\Qlb{fig1}}{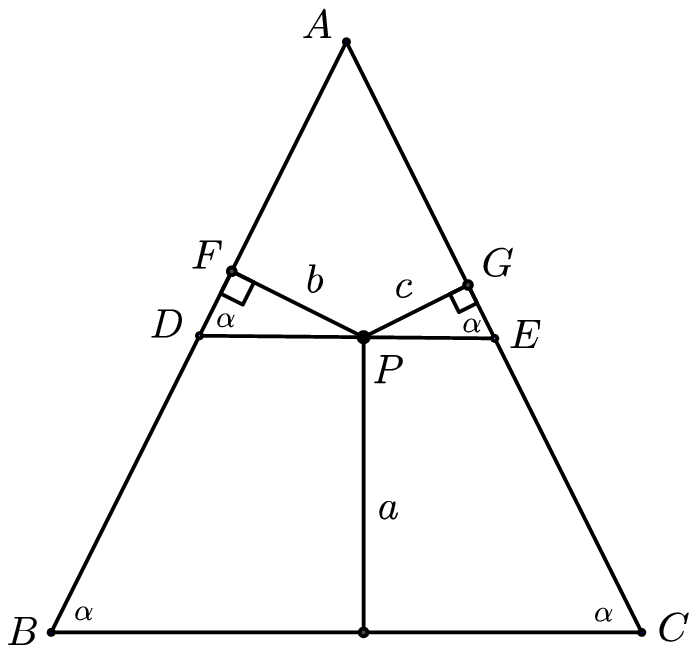}{\special{language
"Scientific Word";type "GRAPHIC";display "USEDEF";valid_file "F";width
3.9807in;height 2.6472in;depth 0pt;original-width 6.4013in;original-height
3.7395in;cropleft "0.1641";croptop "1";cropright "0.8111";cropbottom
"0";filename 'isoceles.eps';file-properties "XNPEU";}}

\subsection*{Scalene triangle}

Similarly, by Theorem \ref{equilateral}, the locus of points $T_{k}(\Delta )$
for a scalene triangle $\Delta $ is a line segment, provided $k$ ranges
between the lengths of the smallest and the largest altitudes of the
triangle. Otherwise, $T_{k}(\Delta )$ will be empty. Indeed, the connection
with linear programming allows us to deduce the result. The convex polygon
is taken to be the feasible region, and the distance sum function $\mathcal{V%
}$ corresponds to the objective function. The parallel line segments, on
which $\mathcal{V}$ is constant, correspond to the isoprofit lines. The
mathematical theory behind linear programming states that an optimal
solution to any problem will lie at a corner point of the feasible region.
If the feasible region is bounded, then both the maximum and the minimum are
attained at corner points. Now, the distance sum function $\mathcal{V}$ is a
linear continuous function in two variables. The values of $\mathcal{V}$ at
the vertices of the feasible region, which is the triangle $\Delta ,$ are
exactly the lengths of the altitudes. Moreover, this function attains its
minimum and its maximum at the vertices of the triangle, and ranges
continuously between its extremal values. Hence, \textbf{it takes\ on every
value between its minimum and its maximum}. Therefore, we have the following
proposition.

\begin{proposition}
If $\Delta $ is a scalene triangle and $k$ ranges between the lengths of the
smallest and the largest altitudes of the triangle, then $T_{k}(\Delta )$ is
a line segment, whose end points are on the boundary of the triangle.
\end{proposition}

The question is: How can we determine this segment for a general triangle?
We shall illustrate the method by the following example. Given the vertices
of a triangle, we first compute the equations of the sides then we find the
distances from a general point $(x,y)$ inside the triangle to each of the
sides. In this way we obtain the corresponding distance sum function $%
\mathcal{V}$. Taking $\mathcal{V}=c,$ we get a family of parallel lines
which, for certain values of the constant $c,$ intersect the given triangle
in the desired line segments.

\begin{example}
\label{Ex1}Let $\Delta $ be the right angled triangle with vertices $%
(0,0),(0,3)$ and $(4,0),$ respectively (see Fig. \ref{fig2}). If the
constant sum of distances from the sides is $k,$ $2.4\leq k\leq $ $4,$ then
the locus is a line segment inside the triangle parallel to the line $2x+y=0$%
.

In Fig. \ref{fig2}, $T_{2.4}(\Delta )=\{A\},T_{2.8}(\Delta
)=DG,T_{3.2}(\Delta )=EH,T_{3.6}(\Delta )=FI$ and $T_{4}(\Delta )=\{C\}.$
Note that at the extremal values of $k,$ the segments shrink to a corner
point of the triangle $\Delta .$

Indeed, the smallest altitude of the triangle is $2.4$ and the largest
altitude is $4.$ Hence, by the previous proposition, the\ locus of points is
a line segment. To find this line segment, we compute the distance sum
function $\mathcal{V}.$ The equation of the hypotenuse is $3x+4y=12.$ Hence,%
\begin{equation*}
\mathcal{V}=x+y-\frac{3x+4y-12}{5}=\frac{2}{5}x+\frac{1}{5}y+\frac{12}{5}.
\end{equation*}%
Therefore, the lines $\mathcal{V}=c$ are parallel to the line $2x+y=0$ and
the result follows.
\end{example}

Note that the proofs of Theorems \ref{equilateral} and \ref{convex polygon}
follow the same line.

\FRAME{ftbphFU}{3.7879in}{2.4068in}{0pt}{\Qcb{ $T_{k}(\Delta ),2.4\leq k\leq 
$ $4,$ are line segments inside the triangle parallel to the line $2x+y=0$.}%
}{\Qlb{fig2}}{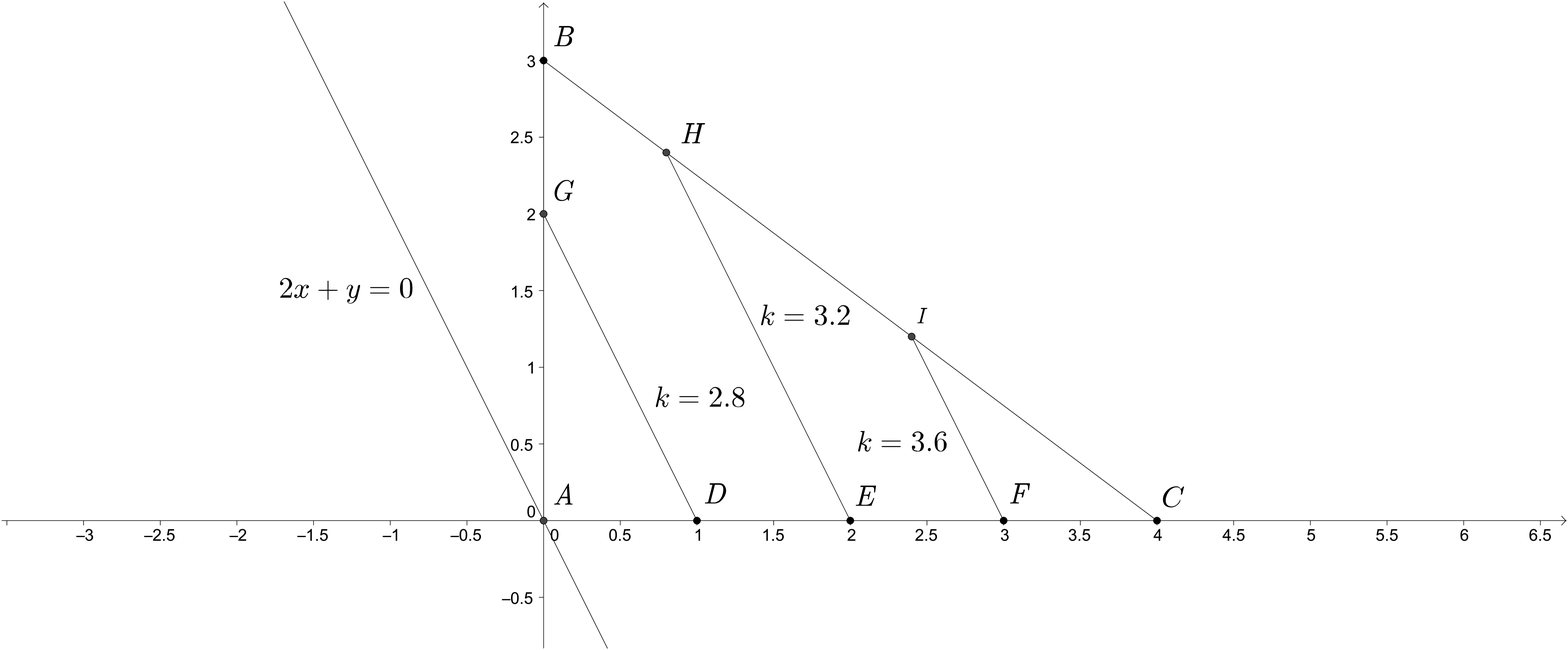}{\special{language "Scientific Word";type
"GRAPHIC";maintain-aspect-ratio TRUE;display "USEDEF";valid_file "F";width
3.7879in;height 2.4068in;depth 0pt;original-width 46.4612in;original-height
19.4764in;cropleft "0.1687";croptop "1";cropright "0.8312";cropbottom
"0";filename 'constant-sum.eps';file-properties "XNPEU";}}

\section*{Constant sum of squares of distances}

Motivated by the previous results, we deal with the second question of
finding the locus of points which have a constant sum of squares of
distances from the sides of a given triangle. Referring to Example \ref{Ex1}%
, $DE^{2}+DF^{2}+DG^{2}=5$ (the number 5 is chosen arbitrary) implies the
following equation of a quadratic curve;

\begin{equation*}
x^{2}+y^{2}+\left( \frac{3x+4y-12}{5}\right) ^{2}=5.
\end{equation*}%
Simplifying, one gets the equivalent equation

\begin{equation*}
34x^{2}+41y^{2}+24xy-72x-96y+19=0.
\end{equation*}%
This is exactly the equation of the ellipse shown in Fig. \ref{fig3}.

\FRAME{ftbphFU}{5.0185in}{3.0528in}{0pt}{\Qcb{Each point $D$ on the ellipse
satisfies $DE^{2}+DF^{2}+DG^{2}=5$}}{\Qlb{fig3}}{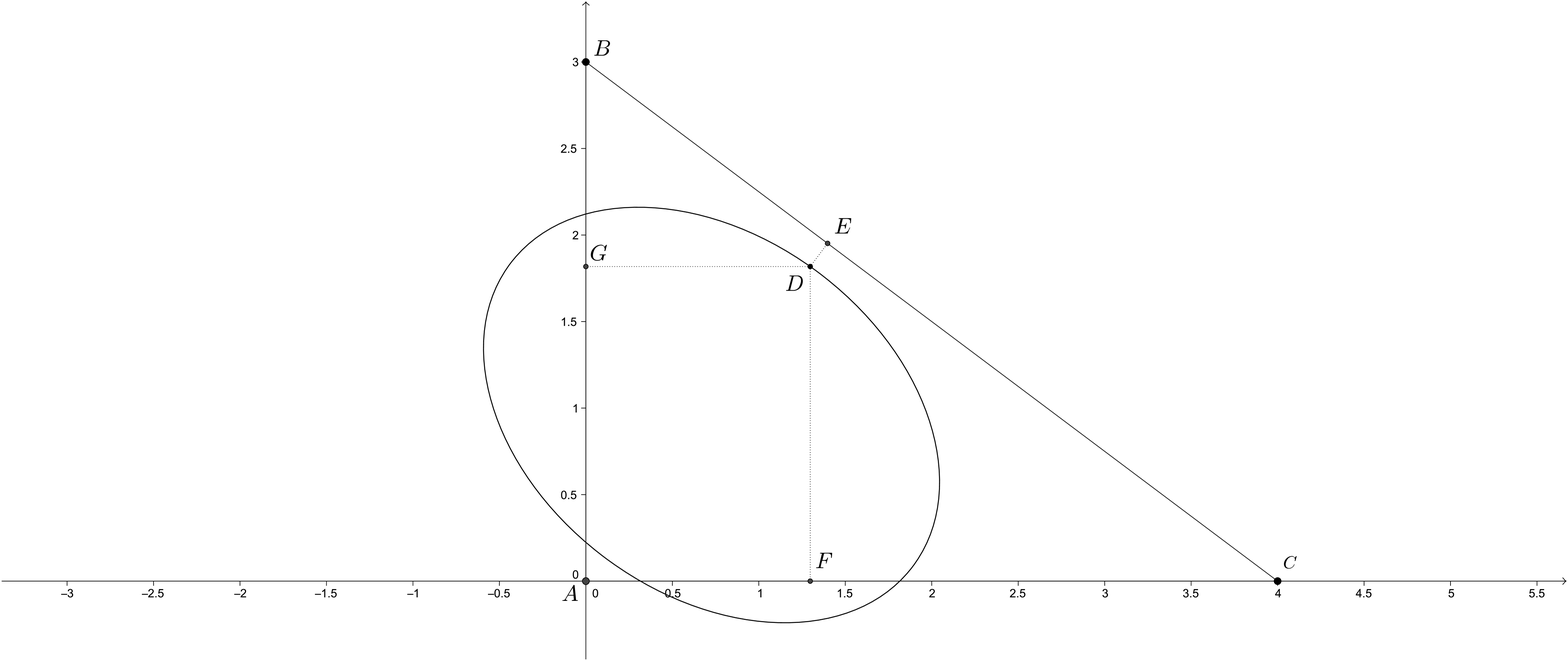%
}{\special{language "Scientific Word";type "GRAPHIC";maintain-aspect-ratio
TRUE;display "USEDEF";valid_file "F";width 5.0185in;height 3.0528in;depth
0pt;original-width 35.376in;original-height 15.0728in;cropleft
"0.2287";croptop "0.9562";cropright "0.8807";cropbottom "0.0155";filename
'square-of-distances-is-5.eps';file-properties "XNPEU";}}

\subsection*{ A new characterization of the ellipse}

In view of the previous example we may prove the following theorem, which
states that among all quadratic curves the ellipse can be characterized as
the locus of points that have a constant sum of squares of distances from
the sides of an appropriate triangle.

\begin{theorem}
The locus $S_{k}(\Delta ABC)$ of points that have a constant sum of squares
of distances from the sides of an appropriate triangle $ABC$ is an ellipse
(and vice versa).
\end{theorem}

\begin{proof}
We shall prove the following two claims:

(a) Given a triangle, the locus of points which have a constant sum of
squares of distances from the sides is an ellipse. These loci, for different
values of the constant, are homothetic ellipses with respect to their common
center (their corresponding axes are proportional with the same factor).

(b) Given an ellipse, there is a triangle for which the sum of the squares
of the distances from the sides, for all points on the ellipse, is constant.

Using analytic geometry, we choose for part (a) the coordinate system such
that the vertices of the triangle lie on the axes. Suppose the coordinates
of the vertices of the triangle are $A(0,a),B(-b,0),C(c,0),$ where $a,b,c>0$
(see Fig. \ref{fig4}). Let $(x,y)$ be any point in the plane and let $%
d_{1},d_{2},d_{3}$ be the distances of $(x,y)$ from the sides of the
triangle $ABC$. It follows that%
\begin{equation*}
\underset{i=1}{\overset{3}{\sum }}d_{i}^{2}=\frac{(ax+cy-ac)^{2}}{a^{2}+c^{2}%
}+\frac{(ax-by+ab)^{2}}{a^{2}+b^{2}}+y^{2}.
\end{equation*}%
Hence, $\underset{i=1}{\overset{3}{\sum }}d_{i}^{2}=k$ (constant) if and
only if the point $(x,y)$ lies on the quadratic curve: 
\begin{equation}
\frac{(ax+cy-ac)^{2}}{a^{2}+c^{2}}+\frac{(ax-by+ab)^{2}}{a^{2}+b^{2}}%
+y^{2}=k.  \label{eq1}
\end{equation}

In general, a quadratic equation in two variables,

\begin{equation}
\mathcal{A}x^{2}+\mathcal{B}xy+\mathcal{C}y^{2}+\mathcal{D}x+\mathcal{E}y+%
\mathcal{F}=0,  \label{eq2}
\end{equation}%
represents an ellipse provided the discriminant $\delta =\mathcal{B}^{2}-4%
\mathcal{AC}<0$ is negative.

To prove our claim, note that from equation (\ref{eq1}) we have:%
\begin{equation}
\mathcal{A}=\frac{a^{2}}{p}+\frac{a^{2}}{q},\mathcal{B}=\frac{2ac}{p}-\frac{%
2ab}{q},\mathcal{C}=\frac{c^{2}}{p}+\frac{b^{2}}{q}+1  \label{eq3}
\end{equation}%
where, $p=$ $a^{2}+c^{2}$ and $q=$ $a^{2}+b^{2}.$ Therefore, 
\begin{equation*}
\delta =\mathcal{B}^{2}-4\mathcal{AC}=-4\frac{a^{2}}{pq}\left(
b^{2}+2bc+c^{2}+p+q\right) \text{,}
\end{equation*}%
and the result follows.

Now, we prove part (b). Applying rotation or translation of the axes we may
assume that the ellipse has a canonical form $\frac{x^{2}}{\alpha ^{2}}+%
\frac{y^{2}}{\beta ^{2}}=1,$ $\alpha \geq \beta >0.$ If we find positive
real numbers $a,b$ with $\alpha ^{2}=a^{2}+3b^{2},\beta ^{2}=2a^{2},$ then
the isosceles triangle with vertices $A^{\prime }(0,a-l),B^{\prime
}(-b,-l),C^{\prime }(b,-l),$ gives the required property in part (b), where $%
l=\frac{2ab^{2}}{a^{2}+3b^{2}}.$

Indeed, because of the symmetry of the ellipse, we choose first an isosceles
triangle with vertices $A(0,a),B(-b,0),C(b,0)$ and compute the sum of the
squared distances from its sides. Substituting $c=b$ in equation (\ref{eq1}%
), gives the equation of the locus $S_{k}(\Delta ABC)$ as

\begin{equation*}
\frac{(ax+by-ab)^{2}}{a^{2}+b^{2}}+\frac{(ax-by+ab)^{2}}{a^{2}+b^{2}}%
+y^{2}=k.
\end{equation*}%
Equivalently, we get a translation of a canonical ellipse:

\begin{equation}
\frac{x^{2}}{a^{2}+3b^{2}}+\frac{(y-\frac{2ab^{2}}{a^{2}+3b^{2}})^{2}}{2a^{2}%
}=\frac{(a^{2}+b^{2})k-2a^{2}b^{2}}{2a^{2}(a^{2}+3b^{2})}+\frac{2b^{4}}{%
(a^{2}+3b^{2})^{2}}.  \label{eq3.5}
\end{equation}%
Thus, it is enough to take 
\begin{equation*}
\frac{(a^{2}+b^{2})k-2a^{2}b^{2}}{2a^{2}(a^{2}+3b^{2})}+\frac{2b^{4}}{%
(a^{2}+3b^{2})^{2}}=1,
\end{equation*}%
and therefore,%
\begin{equation}
k=\frac{2a^{2}(a^{4}+7a^{2}b^{2}+10b^{4})}{(a^{2}+b^{2})(a^{2}+3b^{2})}.
\label{eq4}
\end{equation}

Translating downward by $l=\frac{2ab^{2}}{a^{2}+3b^{2}},$ we get that the
ellipse $\frac{x^{2}}{a^{2}+3b^{2}}+\frac{y^{2}}{2a^{2}}=1$ is the locus of
points, $S_{k}(\Delta A^{\prime }B^{\prime }C^{\prime }),$ which have a
constant sum of squares of distances from the sides of the triangle with
vertices $A^{\prime }(0,a-l),B^{\prime }(-b,-l),C^{\prime }(b,-l).$
Moreover, this constant is given by equation (\ref{eq4}).
\end{proof}

\FRAME{ftbphFU}{4.3145in}{2.2615in}{0pt}{\Qcb{The locus is an ellipse}}{\Qlb{%
fig4}}{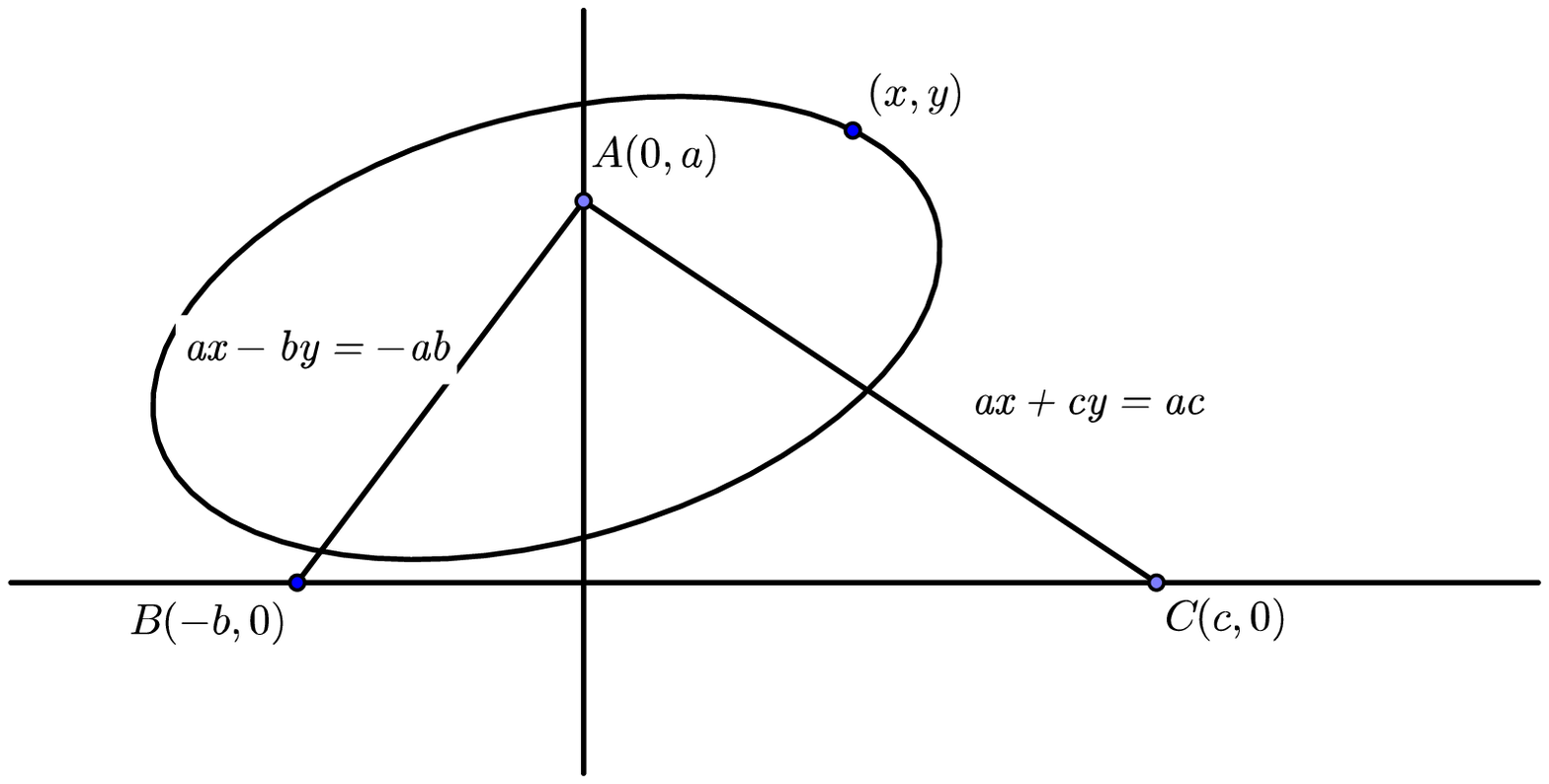}{\special{language "Scientific Word";type
"GRAPHIC";maintain-aspect-ratio TRUE;display "USEDEF";valid_file "F";width
4.3145in;height 2.2615in;depth 0pt;original-width 8.0652in;original-height
4.7063in;cropleft "0.0825";croptop "0.9036";cropright "0.7603";cropbottom
"0.0778";filename 'ellipse1.eps';file-properties "XNPEU";}}\ 

Table \ref{table1} includes two examples that demonstrate Theorem 3.

%TCIMACRO{\TeXButton{B}{\begin{table}[H] \centering}}%
%BeginExpansion
\begin{table}[H] \centering%
%EndExpansion
\begin{tabular}{ccccccc}
\hline
$a$ & $b$ & $\alpha ^{2}=a^{2}+3b^{2}$ & $\beta ^{2}=2a^{2}$ & $\text{Ellipse%
}$ & $k$ & $\text{Figure}$ \\ \hline
$1$ & $1$ & $4$ & $2$ & $\frac{x^{2}}{4}+\frac{y^{2}}{2}=1$ & $\frac{9}{2}$
& $\ref{fig5}$ \\ 
$\sqrt{3}$ & $1$ & $6$ & $6$ & $x^{2}+y^{2}=6$ & $10$ & $\ref{fig6}$ \\ 
\hline
\end{tabular}%
\caption{Examples}\label{table1}%
%TCIMACRO{\TeXButton{E}{\end{table}}}%
%BeginExpansion
\end{table}%
%EndExpansion

Notice that, though a given triangle and a given constant $k$ yield at most
one ellipse, the same ellipse can be obtained using different triangles.
This can be easily seen in Fig. \ref{fig5}. By the above computations, the
ellipse $\frac{x^{2}}{4}+\frac{y^{2}}{2}=1$ is the locus of points $%
S_{4.5}(\Delta ),$ where $\Delta $ is the triangle with vertices $A^{\prime
}=$ $(0,0.5),B^{\prime }=(-1,-0.5)$ and $C^{\prime }=(1,-0.5).$ If we
reflect the whole figure across the $x-$axis, then the same ellipse $\frac{%
x^{2}}{4}+\frac{y^{2}}{2}=1$ will be the locus of points $S_{4.5}(\widetilde{%
\Delta }),$ where $\widetilde{\Delta }$ is the reflection of $\Delta $
across the $x-$axis with vertices: $(0,-0.5),(-1,0.5)$ and $(1,0.5).$

Notice also that, in the second example, the locus of points is a circle. In
general, the locus of points is a circle exactly when, in equation (\ref{eq2}%
), 
\begin{equation*}
\mathcal{A}=\mathcal{C}\text{ and }\mathcal{B}=0.
\end{equation*}%
Equivalently, from (\ref{eq3}) we have;%
\begin{equation}
\frac{a^{2}}{p}+\frac{a^{2}}{q}=\frac{c^{2}}{p}+\frac{b^{2}}{q}+1\text{ and }%
\frac{2ac}{p}-\frac{2ab}{q}=0.  \label{eq5}
\end{equation}%
Substituting the values of $p$ and $q$ and simplifying we get;%
\begin{equation*}
\frac{2ac}{p}-\frac{2ab}{q}=0\Leftrightarrow (a^{2}-bc)(c-b)=0.
\end{equation*}%
Consequently, two cases have to be considered:

First, $a^{2}=bc,$ in which case the triangle $\Delta ABC$ is right angled.
This case does not occur, since the conditions in (\ref{eq5}) lead to a
contradiction as follows: $\frac{2ac}{p}-\frac{2ab}{q}=0$ implies $\frac{c}{p%
}=\frac{b}{q}$ or equivalently $\frac{bc}{p}=\frac{b^{2}}{q}.$ Since, $%
a^{2}=bc$ we get $\frac{a^{2}}{p}=\frac{b^{2}}{q}.$ Hence, the first
condition in (\ref{eq5}) simplifies into\ $\frac{a^{2}}{q}=\frac{c^{2}}{p}%
+1. $ This equation together with the relations $a^{2}=bc$ and $\frac{c}{p}=%
\frac{b}{q},$ yield a contradiction.

Second, $c=b,$ in which case the triangle $\Delta ABC$ is isosceles. In this
case, we have $p=q$. Substituting in the first condition of (\ref{eq5}) we
get $a^{2}+3b^{2}=2a^{2},$ which is equivalent to $a=\sqrt{3}b.$ Hence, the
vertices of the triangle are $A(0,\sqrt{3}b),B(-b,0)$ and $C(b,0)$. This
implies that the triangle is equilateral.

Therefore, we have the following result.

\begin{conclusion}
The locus $S_{k}(\Delta )$ of points that have a constant sum of squares of
distances from the sides of a given triangle $\Delta $ is a circle if and
only if the triangle $\Delta $ is equilateral. \ \ \ 
\end{conclusion}

\FRAME{ftbphFU}{3.723in}{2.7959in}{0pt}{\Qcb{The locus is an ellipse with $%
k=4.5$}}{\Qlb{fig5}}{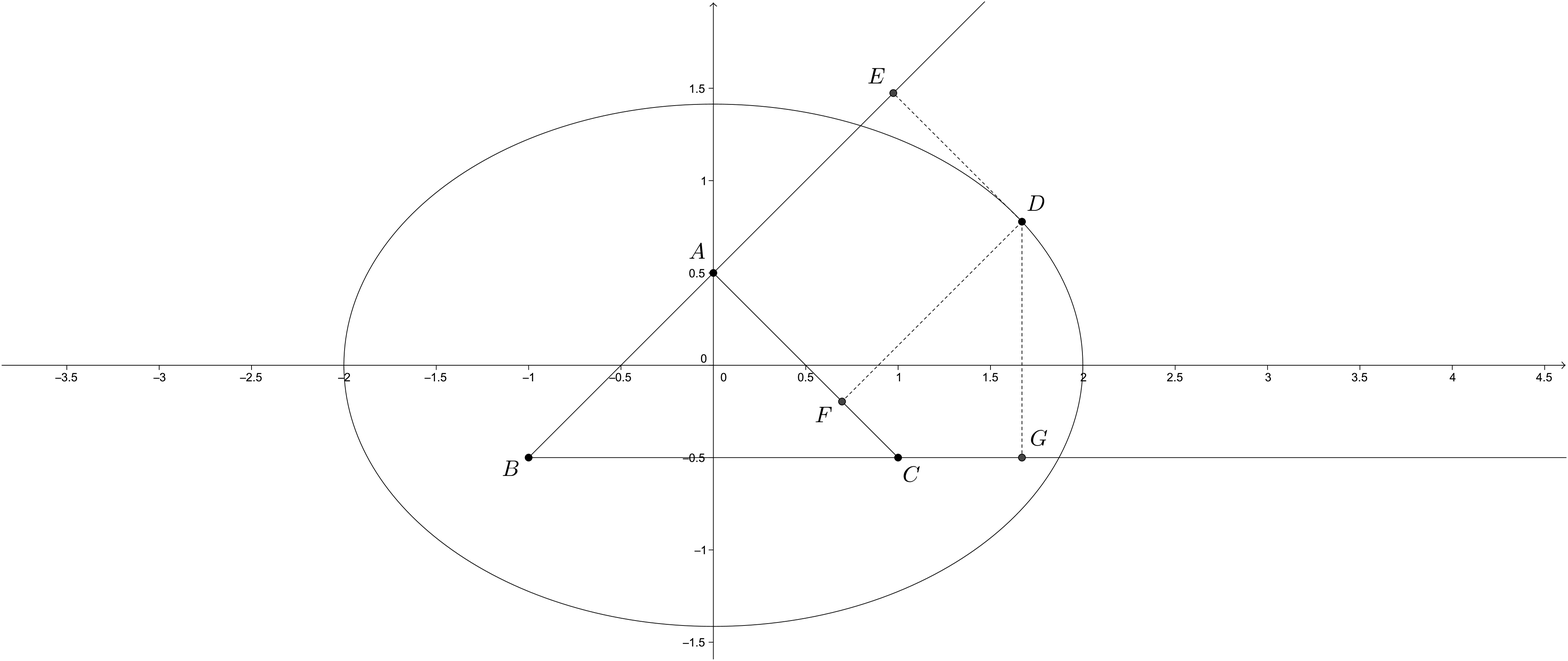}{\special{language
"Scientific Word";type "GRAPHIC";maintain-aspect-ratio TRUE;display
"USEDEF";valid_file "F";width 3.723in;height 2.7959in;depth
0pt;original-width 16.8647in;original-height 7.1857in;cropleft
"0.1750";croptop "1";cropright "0.7439";cropbottom "0";filename
'square-of-distances-k-is-4.eps';file-properties "XNPEU";}}

\FRAME{ftbphFU}{3.4463in}{2.949in}{0pt}{\Qcb{The locus is a circle with $%
k=10 $}}{\Qlb{fig6}}{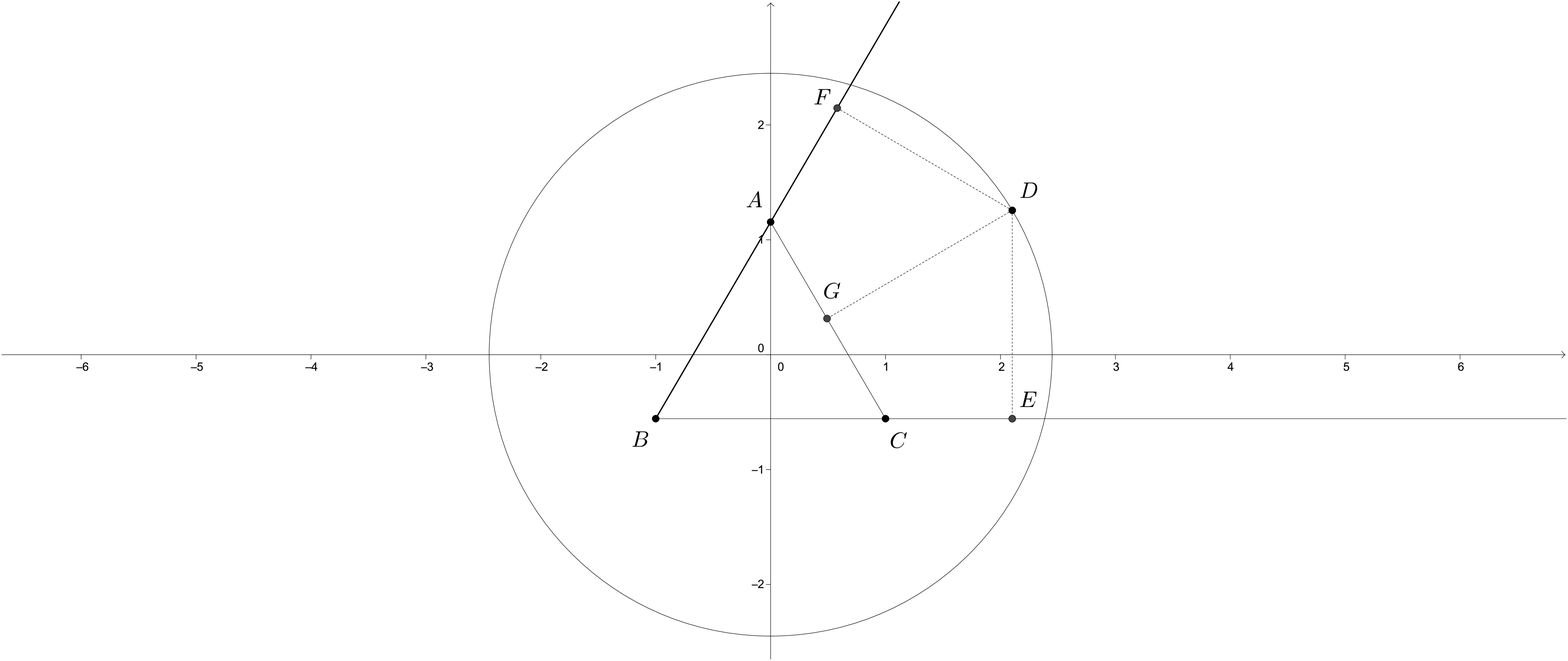}{\special{language
"Scientific Word";type "GRAPHIC";maintain-aspect-ratio TRUE;display
"USEDEF";valid_file "F";width 3.4463in;height 2.949in;depth
0pt;original-width 27.1205in;original-height 11.5513in;cropleft
"0.2234";croptop "1";cropright "0.7218";cropbottom "0";filename
'square-of-distances-k-is-10.eps';file-properties "XNPEU";}}

\subsection*{Minimal sum of squared distances}

Another interesting question is to find the minimal sum of squared
distances, from the sides of a given triangle. For the isosceles triangle
with vertices $A(0,a),B(-b,0)$ and $C(b,0),$ equation (\ref{eq3.5})
indicates that, as the right hand side approaches $0,$ the ellipse
degenerates to one point: $(0,\frac{2ab^{2}}{a^{2}+3b^{2}}).$ Thus, the
locus of points is defined exactly when%
\begin{equation*}
\frac{(a^{2}+b^{2})k-2a^{2}b^{2}}{2a^{2}(a^{2}+3b^{2})}+\frac{2b^{4}}{%
(a^{2}+3b^{2})^{2}}\geq 0.
\end{equation*}%
Equivalently, 
\begin{equation*}
k\geq \frac{2a^{2}b^{2}}{a^{2}+3b^{2}}.
\end{equation*}%
Therefore the following consequence holds.

\begin{conclusion}
\label{Con2}For the isosceles triangle with vertices $A(0,a),B(-b,0)$ and $%
C(b,0),$ the minimal sum of squared distances from the sides is $\frac{%
2a^{2}b^{2}}{a^{2}+3b^{2}}$ attained at the point $(0,\frac{2ab^{2}}{%
a^{2}+3b^{2}}),$ inside the triangle.
\end{conclusion}

When $a^{2}=3b^{2}$ then the triangle is equilateral with vertices $A(0,%
\sqrt{3}b),B(-b,0)$ and $C(b,0).$ In this case, the minimal sum of squared
distances is $b^{2}$ attained at the point $(0,\frac{\sqrt{3}}{3}b)$ which
is exactly the incenter of\ the equilateral triangle. In Fig. \ref%
{fig-minimal-k-circle}, $b=1$ and the loci of points are circles which
degenerate to the incenter of the equilateral triangle as $k$ approaches $1.$

\FRAME{ftbphFU}{4.3422in}{2.8314in}{0pt}{\Qcb{The circles degenerate to the
incenter $G$ of the equilateral triangle, where the minimum of $k$ is $1.$}}{%
\Qlb{fig-minimal-k-circle}}{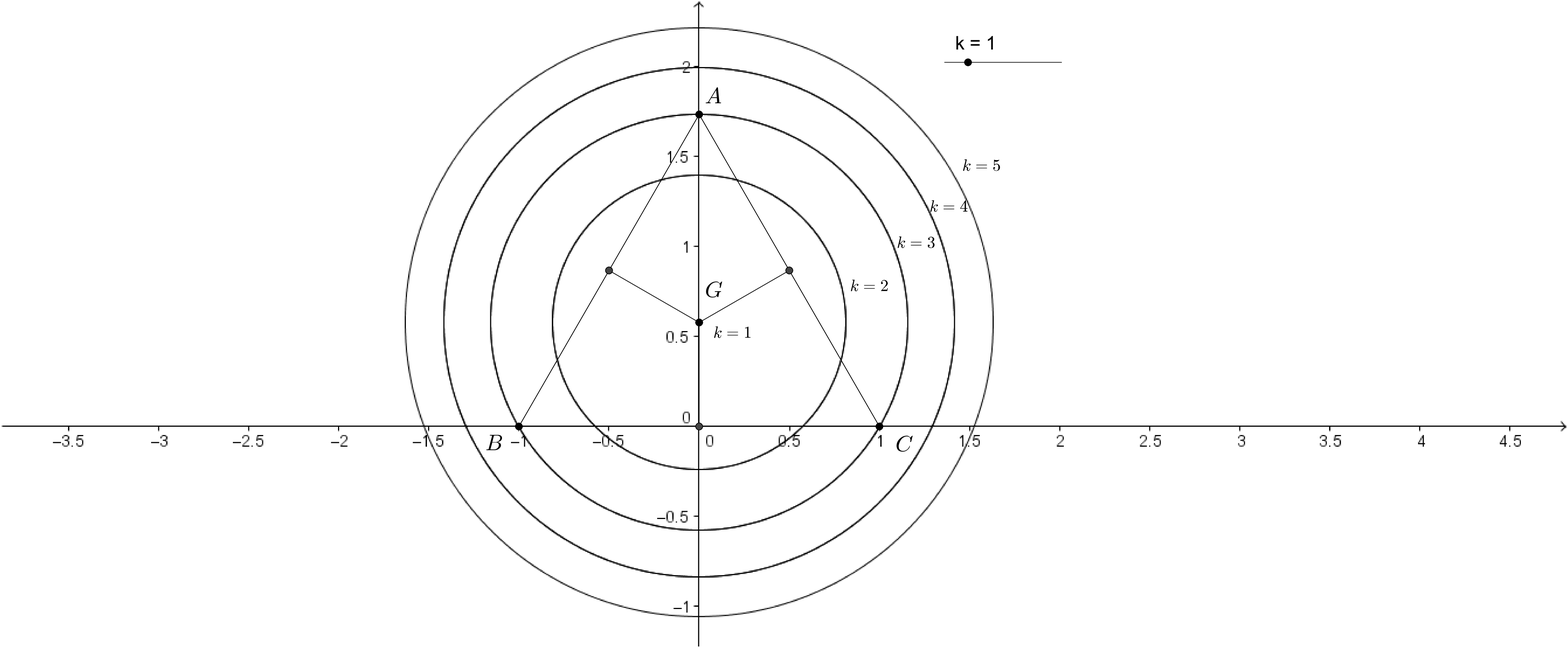}{\special{language
"Scientific Word";type "GRAPHIC";maintain-aspect-ratio TRUE;display
"USEDEF";valid_file "F";width 4.3422in;height 2.8314in;depth
0pt;original-width 34.193in;original-height 14.3213in;cropleft
"0.1340";croptop "1";cropright "0.7785";cropbottom "0";filename
'minimal-k-circle.eps';file-properties "XNPEU";}}

\section*{Concluding remarks}

Other related results are the following: Kawasaki \cite[p. 213]{Kaw}, with a
proof without words, used only rotations to establish Viviani's theorem.
Polster \cite{Pol}, considered a natural twist to a beautiful one-glance
proof of Viviani's theorem and its implications for general triangles.

The restriction of the locus, in the first question, to subsets of the
closed triangle is intended to avoid the use of the signed distances. These
signed distances, when considered, allows us to search loci of points
outside the triangle for "large" values of the constant $k.$ The reader is
encouraged to do some examples.

Theorem \ref{convex polygon}, allows us to address Question 1 to any convex
polygon in the plane.

The question about loci of points, with constant sum of squares of
distances, can be generalized to any polygon or any set of lines in the
plane. In this case, one should distinguish whether all the lines are
parallel or not. The reader is encouraged to work out some examples using
GeoGebra.

Conclusion \ref{Con2} can be restated for general triangles, by a change of
the coordinate system. This demands familiarity with the following theorem
(which is beyond our discussion): Every real quadratic form $q=X^{T}AX$ with
symmetric matrix $A$ can be reduced by an orthogonal transformation to a
canonical form.

Finally, this characterization of the ellipse was exploited to build an
algorithm for drawing ellipses using GeoGebra (see \cite{abb0}).

\ 

{\large Acknowledgement}: \textit{The author is indebted to the referees,
who read the manuscript carefully, and whose valuable comments concerning
the style, the examples, the theorems and the proofs improved substantially
the exposition of the paper. Special thanks are also due to the editor for
his valuable remarks}.

\textit{This work \textit{is part of research which }was supported by Beit
Berl College Research Fund.}

\textit{\ }

\bigskip

\ ELIAS ABBOUD (MR author ID: 249090) received his D.Sc from the Technion,
Israel. Since 1992 he has taught Mathematics at Beit Berl College. He serves
as the Math Chair in the Arab Academic Institution within the Faculty of
Education of Beit Berl College.

\end{document}